\documentclass[11pt]{article}
\usepackage[english]{babel}
\usepackage{amsthm, amsmath, amssymb, amsbsy, amsfonts, latexsym, euscript}
\usepackage [latin1]{inputenc}
\usepackage{graphicx}
\theoremstyle{plain}
\newtheorem{theorem}{Theorem}[section]

\newtheorem{lemma}[theorem]{Lemma}
\newtheorem{corollary}[theorem]{Corollary}
\theoremstyle{definition}
\newtheorem{definition}[theorem]{Definition}

\newtheorem{remark}[theorem]{Remark}

\begin{document}

\title{Galois theory for finite fields}

\author{Askold Khovanskii}

\maketitle

\begin{abstract}
This note presents  Galois  theory for finite fields. It was written as a handout for the  MAT401 course ``Polynomial equations and fields'' taught at the University of Toronto  in Spring 2026. We use without proofs some basic properties of finite fields and of finite field extensions which we already covered in class. Firstly, we  describe an extension  $K\subset F$ of a finite field $K$ of a given degree $n$. We show that the set of all intermediate fields for this extension is in one-to-one correspondence with the set of all divisors $k$ of the degree $n$. Then we  describe the Galois group of this extension which is the cyclic group of order $n$. The set of subgroups of this group also is in one-to-one correspondence  with the set of all divisors $k$ of the degree $n$. It allows us to prove the Galois correspondence for that extension. In the last section, we state basic theorems of Galois theory for arbitrary fields which will be proven later in the course.
\end{abstract}

\section{Introduction}\label{intro}

Below, we present  Galois theory for finite fields. We use without proofs some
 basic  properties of  field extensions which we study in class. In particular, it includes  basic properties of finite fields, listed in the  section \ref{prelim}. We use
 splitting fields of polynomials over a given field.  We also use the following elementary but very useful  Degree formula: if $K\subset B\subset F$ is a chain of finite extensions, then the degree of the extension $K\subset F$ is the product of degrees of the extension $K\subset B$ and of the extension $B\subset F$.

We  recall   some simple properties of cyclic groups. In the theory  of finite fields, cyclic groups appear in two  different ways: on the one hand, the multiplicative group of a finite field is a cyclic group, on the other hand, the Galois group of  finite extension of any finite field  is a cyclic group.

In the section \ref{sectwo},
we recall the famous  Chinese  Remainder  Theorem and its versions.

In the section \ref{secthre}, we show that the multiplicative group of roots of unity in an arbitrary  field is a cyclic group.

A description of the multiplicative group $U(n)$ of invertible elements in the ring $\Bbb Z \mod n$ can be reduced to the case when $n$ is a power of a prime number $p$.  We describe groups $U(p^m)$ in the section \ref{u}.

In the section \ref{extrasec},
 we show that the multiplicative group of an arbitrary finite field is a cyclic group. As an application, we prove a version of the Primitive Element Theorem for finite fields.

In the section \ref{secfour},
we explore a  finite extension $K\subset F$ of a finite field $K$ of  any given degree $n$. We show that such an extension is  unique. We  describe all intermediate fields for that extension.

In the section \ref{five},
we study automorphisms of finite fields. In particular, we show that any automorphism of a finite field $F$ of characteristic $p$ is a power of the Frobenius map $\Phi:F\to F$, where $\Phi(a)=a^p$.

In the section \ref{six},
we define the Galois group $G(F,K)$ of a finite extension $K\subset F$, where $K$ and $F$ are finite fields. We show that the group $G(F,K)$ is well defined and describe this group.

In the section \ref{seven},
we discuss the Galois correspondence for the extension $K\subset F$ of degree $kn$  of a finite field $K$. It provides a one-to-one correspondence
between the set of all subgroups  of the Galois Group $G(F,K)$ of the  extension  and  the set of all intermediate subfields  of this extension.

In the section \ref{seceight},
we present without proof basic theorems of Galois theory  for arbitrary fields. Proofs of these theorems will be presented in the course later.

\section{Preliminaries  on finite fields}\label{prelim}
In this section, we list  simple results on finite fields which we studied in class. We will use these results without proof. Such proofs can be found in many places, see for example a textbook \cite{Rotman} which we use in the course.

\begin{enumerate}
\item Each finite field $F$ has  characteristic $p$, where $p$ is some prime number,  i.e., $F$ contains the subfield field $\Bbb Z_p$. The field $F$ is a finite extension of the field $\Bbb Z_p$. The  degree $n$ of this extension together with the characteristic $p$ totally determines the finite field $F$ which is   called {\it the Galois field}    $\Bbb F_q$, where $q=p^n$. The field  $\Bbb F_q$ can be considered as  an $n$-dimensional vector space over the field $\Bbb Z_p$. It contains $q=p^n$ elements.

\item A field $K$ contains the field $\Bbb F_q$ if and only if $K$ has characteristic $p$  and the polynomial $x^q-x$  has exactly $q=p^n$ roots in $K$. Moreover, the set of elements in the subfield $\Bbb F_q\subset K$ is equal to the set of all roots in $K$ of the polynomial  $x^q-x$.
\item For any field $K$ of characteristic $p$ the map $\Phi:K\to K$, which sends $x$ to $x^p$, is called the {\it Frobenius homomorphism}. The  homomorphism $\Phi$ may not   be an one-to-one map: some elements of $K$  could have no preimages under the map $\Phi$. But if $K$ is a finite field, then $\Phi:K\to K$ is  an onto map.

 \item  The statement 2) means that the  $\Bbb F_q$ belongs to $K$ if and only if the $q$-th power $\Phi^q:K\to K$ of the  Frobenius homomorphism has exactly $q=p^n$ fixed points, and the  set $\Bbb F_q\subset K$ is equal to the set of all fixed points of $\Phi^n$.
   \end{enumerate}

\section{Around Chinese  Remainder  Theorem}\label{sectwo}

In this section, we recall  the famous  Chinese  Remainder  Theorem and its versions.

Let $m_1,\dots,m_n\in \Bbb Z$ be a collection of natural numbers such that any two different numbers $m_i, m_j$ from the collection are relatively prime. Let $\pi:\Bbb Z\to \Bbb Z/m_1\Bbb Z \oplus\ldots\oplus \Bbb Z/m_n \Bbb Z$ be a ring homomorphism from  $\Bbb Z$ to the direct sum of the factor rings $\Bbb Z/m_i \Bbb Z$ which sends $n\in  \Bbb Z $ to $(n\mod m_1, \ldots, n\mod m_n)$.

\begin{lemma}\label{map} The kernel of the homomorphism $\pi$ is the principal ideal $(d)\subset \Bbb Z$, where $d=m_1\cdot \ldots \cdot m_n$.
\end{lemma}

\begin{proof} The principal  ideal $(m_i)\subset \Bbb Z$ of the number $m_i$ is the kernel of the factorization map $\Bbb Z\to \Bbb Z/ m_i \Bbb Z$. The kernel of $\pi$ is the ideal $\cap_{1\leq i\leq n} (m_i)\subset \Bbb Z$. Since the numbers $m_i$ are relatively prime, we have
$\cap_{1\leq i\leq n} (m_i)=(m_1\cdot \ldots \cdot m_n)=(d)$.
\end{proof}

\begin{corollary}\label{ringiso} If the natural numbers $m_1,\dots,m_n\in \Bbb Z$ are relatively prime, then the ring $\Bbb Z/ d \Bbb Z$, where $d=m_1\cdot \ldots \cdot m_n$, is isomorphic  to the direct sum $\oplus_{1\leq i\leq n} \Bbb Z/ m_i \Bbb Z$ of the rings $\Bbb Z/m_i \Bbb Z$.
\end{corollary}

\begin{proof} By Lemma \ref{map}, the homomorphism $\pi$ provides an embedding of the finite ring $\Bbb Z/d \Bbb Z$ to the direct sum $\oplus_{1\leq i\leq n} \Bbb Z/ m_i \Bbb Z$ of the finite rings $\Bbb Z/m_i \Bbb Z$. Since the rings $\Bbb Z/d \Bbb Z$ and $\oplus_{1\leq i\leq n} \Bbb Z/ m_i \Bbb Z$  contain the same number $d$ of elements, the embedding  $\pi$ is an isomorphism.
\end{proof}

\begin{corollary}\label{invert} If the natural numbers $m_1,\dots,m_n$ are  pairwise relatively prime, then  the multiplicative group $U(d)$ of invertible elements in the ring $\Bbb Z/ d \Bbb Z$, where $d=m_1\cdot \ldots \cdot m_n$, is isomorphic  to the direct product $\prod_{1\leq i \leq n} U(m_i)$ of the multiplicative groups $U(m_i)$ of invertible elements in the ring $\Bbb Z/ m_i \Bbb Z$.
\end{corollary}
\begin{proof} Corollary \ref{invert} follows from Corollary \ref{ringiso}.
\end{proof}

Below, in section \ref{u}, we will provide a description of the  group $U(p^m)$ for any prime number $p$ and any natural natural $m$.

 Now we will compute the number of elements in the group $U(p^m)$ and, more generally, the number of elements  in the group $U(n)$ for an arbitrary natural number $n$.

\begin{corollary} Let $n=p_1^{k_1}\cdot\ldots\cdot p_m^{k_m}$ be the prime factorization of a natural number $n$. Then the group $U(n)$ contains $\prod_{1 \leq i \leq n} p_i^{m_i-1} (p_i-1)$ elements.
\end{corollary}

\begin{proof} By Corollary \ref{invert}, it is enough to prove the above corollary for $n=p^m$. The number of elements in the group $U(n)$ is equal to the number of elements in $\Bbb Z/ (n)$ which are relatively prime with $n$. For $n=p^m$ the number of elements in $\Bbb Z/ (p^m)$, not relatively prime  with $p^m$,  is equal to $p^m:p=p^{m-1}$. Indeed, a number is not relatively prime with $p^m$ if and only if it is divisible by $p$.
\end{proof}

\begin{corollary}[Chinese  Remainder  Theorem] \label{china} If the natural numbers $m_1,\dots,m_n\in \Bbb Z$ are pairwise relatively prime, then the cyclic group $\Bbb Z_d $ of  order  $d=m_1\cdot \ldots \cdot m_n$ is isomorphic  to the direct sum $\oplus_{1\leq i\leq n} \Bbb Z_{ m_i}$ of the cyclic groups $\Bbb Z_ {m_i}$ of orders $m_i$.
\end{corollary}

\begin{proof} Corollary \ref {china} also follows from Corollary \ref{ringiso}.
\end{proof}

\section{Multiplicative group of  roots of unity}\label{secthre}

In this section, we recall basic  properties of cyclic groups, and applications of these properties to studies of roots of unity in an arbitrary field.

\begin{lemma}\label{cyclic1} Any subgroup of a cyclic group is also a cyclic group. The cyclic group $\Bbb Z_n$ of order $n$ contains a subgroup of order $k$ if and only if $k$ divides $n$. Moreover, a subgroup of order $k$ is unique and it consists of all elements $x\in \Bbb Z_n\sim \Bbb Z/\mod n$  satisfying the relation $k x=0 \mod n$.
\end{lemma}

\begin{proof} Any nonzero subgroup $G\subset \Bbb  Z$ is a cyclic group generated by the smallest positive element in $G$. The group $\Bbb Z_n$ is a factor group of $\Bbb Z$, thus each nonzero subgroup of $\Bbb Z_n$ is a cyclic group, whose order $k$ divides $n$, which contains $n:k$ elements. All elements $x\in G$   satisfy the equation $k x=0/ \mod n$. But the number of solutions of this equation is also equal to $n:k$. Thus, $G$ coincides with the set of solutions $kx=0\mod n$ and, therefore, it is unique.
\end{proof}

\begin{theorem}\label{cyclic2} A finite commutative  group $G$ is cyclic if and only if for any natural number  $k$ it has  at most one cyclic subgroup of order $k$.
\end{theorem}

\begin{proof} According to Lemma \ref{cyclic1}, a cyclic group  for any  natural number $k$ has at most one cyclic subgroup of order $k$.
On the other hand, any finite  commutative group $G$ is representable as a direct sum of cyclic groups of some orders $m_1,\dots,m_n$. If some different  numbers $m_i$ and $m_j$ have a nontrivial common
divisor $d$, then both cyclic groups $\Bbb Z_{m_i}$ and $\Bbb Z_{m_j}$ have cyclic subgroups of order $d$. In that case, according to Lemma \ref{cyclic1}, the group $G$ is not a cyclic group. If the numbers $m_1,\dots, m_n$ are pairwise relatively  prime, then, by Corollary \ref{china}, the group $G$ is cyclic.
\end{proof}

For any field $K$ and for any natural number $n$ let $K^*_n\subset K$ be the set of all roots of unity of order $n$ in the field $K$, i.e., $K^*_n$ is the set  of all solution in $K$ of the equation  $x^n-1=0$.

\begin{theorem} For any field $K$  and for any natural number $n$ the set $K^*_n \subset K$ form a subgroup in the  multiplicative group $K^*$ of the field $K$. Moreover,  $K^*_n$ is a cyclic group.
\end{theorem}

\begin{proof} If $x^n=1$ and $y^n=1$, then $(xy)^n=1$. If $x^n=1$, then $(x^{-1})^n=1$. Thus, $K^*_n$ is a subgroup in the  multiplicative group $K^*$ of the field $K$.
The group $K^*_n$ is a finite commutative group. For any natural number $k$  the equation $x^k=1$ has  at most $k$ solutions in $K$. Thus, the group $K^*_n$ could contain at most one cyclic subgroup of order $k$: such a subgroup exists if and only if the polynomial $x^k-1$ has exactly $k$ different roots in the field $K$.  Thus, by Theorem \ref{cyclic2},  $K^*_n$ is a cyclic group.
\end{proof}

\section{Multiplicative  group $U(p^n)$}\label{u}

In this section, we will show that for any prime number $p\neq 2$ and any natural number $n$ the multiplicative group $U(p^n)$ is a cyclic group. We also  will describe groups $U(2^n)$ for any natural number  $n$. First of all, the group $U(p)$ for any prime $p\neq 2$  is a cyclic group,  since it is the multiplicative group of the field $\Bbb Z_p$ which coincides with the group of roots of unity of the order $p-1$ in the field $\Bbb  Z_p$. Below, we describe  groups $U(p^n)$ assuming that $n>1$.

Assume that $p\neq 2$. Consider the ring homomorphism $\pi: \Bbb Z/(p^n)\to \Bbb Z_p$. Its kernel consists of all integers modulo  $p^n$ which are divisible by $p$, thus it consists of all not invertible elements of the ring $\Bbb Z/(p^n)$.    The kernel of the corresponding group  homomorphism  $\pi_1:U(p^n)\to U(p)$ of the multiplicative groups of these rings contains the element $(1+p) \mod p^n$.

\begin{lemma}\label{kernel}  For $p\neq 2$ the kernel of the homomorphism $\pi_1$ is a cyclic group of the order $p^{n-1}$.
\end{lemma}

Lemma \ref{kernel}  is based on the following observation.

\begin{lemma}\label{divisibility} For any prime $p\neq 2$ and for any $n>1$ the order  of the element $1+p \mod p^n$ in the group $U(p^n)$ is equal to $p^{n-1}$.
\end{lemma}

\begin{proof} The element $1+p \mod p^n$ belongs to the kernel of the group homomorphism $\pi_1$ which is  a subgroup in $U(p^n)$ of the order $p^{n-1}$. Thus, the order  $k$ of the element $(1+p) \mod p^n$ divides the number $p^{n-1}$, so $k=p^m$ for some number $m$ satisfying the following  inequalities $1\leq m\leq n-1$.
For any natural number $m$ the identity $(1+p^m)^p \equiv  1 + p^{m+1}  \mod p^{m++2}$ holds. Indeed, $(1+p^m)^p= 1+ p \cdot p^m  + \frac {p(p-1)}{2}p^{2m} +\dots \equiv 1+p^{m+1} \mod p^{m+2}$.
(Note  that  in the proof of above  identities  for $m=1$  one
use that for $p\neq 2$ the number $\frac {p(p-1)}{2}$ is divisible by $p$.)
\end{proof}

\begin{proof}[Proof of Lemma \ref{kernel}]  As we showed above, for any $k<m-1$ the identity $(1+p)^{p^k}\equiv 1+ p^{k+1} \mod p^{k+2}$ holds.  The smallest degree $k$, for which $(1+p)^{p^k}\equiv 1\mod p^n$,  is $k=p^{n-1}$. Thus, the order  of the element $1+p\mod p^n$ is equal to the number of elements in the group $\ker \pi_1$ containing it. So this element generates the group $\ker \pi_1$ and it is cyclic.
\end{proof}

\begin{theorem}\label{ucyclic}
For any prime $p\neq 2$ and for any natural number $n$ the group $U(p^n)$ is a cyclic group.
\end{theorem}

\begin{proof} The homomorphism $\pi_1:U(p^n)\to U(p)$ is a surjective map. Let $a\in U(p)$ be a generator of the cyclic group $U(p)$ and let $b\in U(p^n)$ be such an element that $\pi_1 (b)=a$. Then $b$ generates a cyclic subgroup $G_b\subset U(p^n)$ of  the order $p-1$. One can see that the group $U(p^n)$ is a direct sum of the cyclic subgroup $\ker \pi_1$ of the  order $p^{n-1}$ and the cyclic subgroup $G_b$ of the order $p-1$. Since $p^{n-1}$ and $p-1$ are relatively prime numbers, the group $U(p^n)$ is cyclic.
\end{proof}

The case $p=2$ has  to be considered separately. Consider  the multiplicative group  $U(2^n)$ of the ring $\Bbb Z/  2^n\Bbb Z$  for $n\geq 3$.  Elements of the group $U(2^n)$ may be reduced modulo $8$.
 Any element of $U(2^n)$ is  equal either to $\pm 1\mod 8$ or to $\pm 3 \mod 8$.

\begin{lemma}\label{mod2}
\begin{enumerate}
 \item If $a\equiv \pm 1\mod 8$, then for any $k\geq 1$ we have $a^{2^k}\equiv 1\mod2^{k+3} $.

\item  If $a\equiv \pm 3\mod 8$, then for any $k \geq  1$ we have $a^{2^k}\equiv 2^{k+2} \mod2^{k+3} $.
\end{enumerate}
\end{lemma}

\begin{proof} One can prove this lemma by induction in $k$.

\begin{enumerate}
 \item For $k=1$ we have $a^{2^1}=a^2= (\pm 1+ 8 n)^2= 1\pm 16 n+16n^2= 1\mod 16= 1\mod 2^{1+3}$. If for some $k$ we have $a^{2^k}=1 +2^{k+3} n$, then $a^{2^{k+1}}=1+ 2\cdot 2^{k+3}n + 2^{2k+6}n^2 \equiv 1\mod 2^{(k+1)+3}$. The needed statement is proven.

\item For $k=1$ we have $a^{2^1}=a^2= (\pm 3+ 8 n)^2= 9 \pm 16\cdot 3n +16n^2=1+ 8\mod 16= 1+ 2^{1+2} \mod 2^{1+3}$. If for some $k$ we have $a^{2^k}=1   +2^{k+2} +2^{k+3} n $, then $a^{2^{k+1}}=(1+2^{k+2})^2+ 2\cdot 2^{k+3}n (1+2^{k+2})  +2^{2(k+3)}n^2 \equiv 1 +2^{k+3}\mod 2^{k+4}$.  The needed statement is proven.

\end{enumerate}
\end{proof}

\begin{theorem}\label{ucyclictwo}
 For any natural number $n>2$ the group $U(2^n)$ is isomorphic to the product $\Bbb Z/2^{n-2}\Bbb Z\times \Bbb Z/ 2 \Bbb Z$ of cyclic groups  $\Bbb Z/2^{n-2}\Bbb Z$ and $\Bbb Z/ 2 \Bbb Z$.
\end{theorem}

\begin{proof} The group $U(2^n)$ is a commutative group of the order $2^{n-1}$.  Lemma \ref{mod2} implies that the order of each element of the group is less than $2^{n-1 }$, thus the group is not cyclic. By Lemma \ref{mod2} part 2), the order of the  element $a\equiv 3\mod 2^n$ is equal to
 $2^{n-2}$.  Thus, the group $U(2^n)$ contains a cyclic subgroup isomorphic to the group  $\Bbb Z/2^{n-2}\Bbb Z$. This index of this subgroup in $U(2^n)$ is equal to two. The only commutative group with such properties is the group
$\Bbb Z/2^{n-2}\Bbb Z\times \Bbb Z/ 2 \Bbb Z$.
\end{proof}

One can easily check the following lemma which completes  the description of the groups $U(p^n)$.

\begin{lemma}\label{easy} The group $U(4)$ is isomorphic to the group $\Bbb Z/2 \Bbb Z$. The group $U(2)$ contains only an identity element.
\end{lemma}

\section{Multiplicative group of finite field}\label{extrasec}

In this section, using the  result on roots of unity,
 we show   that the multiplicative group of an arbitrary finite field is a cyclic group. As an application, we prove a version of the Primitive Element Theorem for finite fields.

\begin{theorem}\label{multfield} The multiplicative group $\Bbb F^*_q$ of the Galois field $\Bbb F_q$ with $q=p^n$ is a cyclic group of order $q-1=p^n-1$.
\end{theorem}

\begin{proof} The set of all $q-1$ nonzero elements of the field $\Bbb F_q$ coincides with the set of all roots of the polynomial $x^{q-1}-1$.
\end{proof}

\begin{corollary}\label{finprimel} In any  finite field $\Bbb F_q$, there is an element $a$ such that any nonzero element $b$ of the field is representable in the form $b=a^m$, where $m$ is a natural number.
\end{corollary}

\begin{proof} As $a$ one can take a generator of the cyclic group $\Bbb F^*_q$.
\end{proof}

\begin{remark}
An element $a\in E$ is called a {\it primitive element} for a finite field extension $K\subset E$ if $E$ can be obtained by adjoining $a$ to the field $K$. The above  corollary implies that a finite extension of a finite field always has  a primitive element. One can show that
for any separable finite field extension $K\subset E$ of an arbitrary field $K$ there exists  a primitive element  (see  Primitive Element Theorem in section \ref{seceight}).
  \end{remark}

If $K\subset E$ is a finite field extension, then $E$ can be obtained by adjoining to $K$ a finite set $\{m_1,\dots, m_n\}$ of algebraic elements over $K$. By Primitive Element Theorem, instead of adjoining  to $K$ finitely many elements one can adjoin just one element $s$. Below we will show that for finite fields such a result follows from  properties of the greatest common divisor of a collection of elements on $\Bbb Z$.

Recall that for any collection $m_1,\dots,m_n\in \Bbb Z$ of nonzero natural numbers their greatest  common divisor $s$ is representable in the form  $s=k_1m_1+\dots+k_nm_n$  where  $k_1,\dots,k_n\in \Bbb Z$, and for any $1\leq i\leq n$ the number $m_i$ is representable in the form  $m_i=d_i s$, where $d_i\in \Bbb Z$. One can apply this result to a cyclic group $G$, since $G$ us a  factor group  $G$ of $\Bbb  Z$. In the next lemma  we apply it  to a cyclic  group $G$ in which  the group operation is written in the multiplicative form.

\begin{lemma}\label{gcd}  For any collection $m_1,\dots,m_n\in G$ of element in $G$ there exist an $n$ tuple of integers $k_1,\dots,k_n\in \Bbb Z$ such that the element  $s=m_1^{k_1}\cdot \ldots \cdot m_n^{k_n}$ has the following property: for any $1\leq i\leq n$ the element  $m_i\in G$ is representable in the form  $m_i=s^{q_i}$ for some $q_i\in \Bbb Z$.
\end{lemma}

\begin{theorem} Let $\Bbb Z_p(m_1,\dots,m_n)$ be a subfield in $\Bbb F_q$ for $q=p^k$  generated over $\Bbb Z_p\subset \Bbb F_q$ by elements $m_1,\dots, m_n\in \Bbb F_q$. Then the field $\Bbb Z_p(m_1,\dots,m_n)$ can be  generated over $\Bbb Z_p$ by one element having a form $m_1^{k_1}\cdot \ldots \cdot m_n^{k_n}$.
\end{theorem}

\begin{proof} The multiplicative group of the field  $\Bbb F_q$ is cyclic. Thus, the theorem follows from Lemma \ref{gcd}.
\end{proof}

Which subfields are contained in a finite field $\Bbb F_{p^n}$? The following lemma gives an answer to this question.

\begin{lemma}\label{subfield} The field $\Bbb F_{p^n}$ contains the subfield $\Bbb F_{p^k}$ if and only if $k$ is a divisor of $n$.
\end{lemma}

In the next section, we  prove  Lemma \ref{subfield} using the Degree formula.
Below we prove  Lemma \ref{subfield} using only properties of cyclic groups.
We will need the following lemma.

\begin{lemma} \label{divisibility}Let $a, t,s\in \Bbb Z$ be integers such that $a>1, s>0, t\geq 0$.  Then the number $a^t-1$ is divisible by the number $a^s-1$ if and only if $t=ks$ for some nonnegative $k$.
\end{lemma}

\begin{proof} Let us  show first that if $t\geq s$, then the  numbers $a^t-1$ and $-(a^{t-s}-1)$ are equal modulo $a^s-1$. Indeed, the following identity holds: $(a^{t-s}-1) a^s +a^s-1=a^t-1$. Reducing this identity modulo $a^s-1$, we obtain the needed equality  $(a^{t-s}-1) (-1)=a^t-1 \mod (a^s-1)$.
So if $a^t-1$ is divisible by $a^s-1$ and $t\geq s$, then $a^{t-s}-1$ also is divisible by $a^s-1$. Let us divide $n$ by $s$ with a remainder $n=ks +b$, where $0\leq b<s$. Applying the above statement $k-1$ times, we obtain that if $a^t-1$ is divisible by $a^s-1$, then $a^b-1$ also is divisible by $a^s-1$. But $a^b-1$ for $0\leq b<s$ could divide $a^s-1$ only if $b=0$. Lemma is proven.
\end{proof}

\begin{proof}[Proof of Lemma \ref{subfield} via cyclic group] If $\Bbb  F_{p^n}$ contains $\Bbb F_{p^k}$, then the multiplicative group $F^*_{p^n}$ of field  $\Bbb  F_{p^n}$ contains the multiplicative group of $\Bbb F_{p^k}$. These groups are  cyclic groups of orders $p^n-1$ and $p^k-1$.
Thus, if $\Bbb F^8_{p^k}\subset \Bbb F_{p^n}$ the number $p^k-1$ divides the number $p^n-1$. By Lemma \ref{divisibility}, it implies that $k$ divides $n$. In the opposite direction, if $k$ divides $n$, then the multiplicative group $\Bbb F^*_{p^n}$ contains a cyclic  subgroup   $G$ of order $p^k-1$. Each element $x$  of the group $G$ satisfies the relation $x^{p^k-1}=1$. The set of all elements of $G$ together with the element zero provides $p^k$ solutions of the equation $x^{p^k}-x=0$. These elements form a subfield $\Bbb F_{p^k}$ in the field $\Bbb F_{p^n}$.
\end{proof}

\section{Finite extensions of finite fields and their degrees}\label{secfour}

In this section, we explore   finite extensions $K\subset F$ of a finite field $K$ of  characteristic $p$. We show that there is a unique such extension  of any given degree $n$ and describe all intermediate fields for that extension.

 First of all, any finite field $F$ of characteristic $p$ contains its finite  subfield $\Bbb Z_p$. The degree $n$ of the  extension $\Bbb Z_p\subset F$  completely determines it: the field $F$ is the splitting field of the polynomial $x^q-x$, where $q=p^n$ over the field $\Bbb Z_p$. The extension $\Bbb Z_p\subset F$ is unique in the following sense: if $\Bbb Z_p\subset F_1$ is another splitting field of the polynomial
$x^q-x$, then there is an isomorphism $\pi:F\to F_1$ whose restriction to the field $\Bbb Z_p$ is the identity map.

Similar result holds for an extension of degree $n$ of the finite field $\Bbb F_q$, where $q=p^s$.

\begin{theorem}\label{findielex}  There is a unique extension $\Bbb F_q\subset F$ of degree $n$. The field $F$ is a splitting field of the polynomial  $x^{q^n}-x$ over the field $\Bbb F_q$: if  $\Bbb F_q \subset F_1$ is another splitting field of the polynomial
$x^{q^n}-x$, then there is an isomorphism $\pi:F\to F_1$ whose restriction to the field $\Bbb F_q$ is the identity map. The field $F$ contains $q^n$ elements, and it is the finite field $\Bbb F_{p^{sn}}$.
\end{theorem}

\begin{proof} Consider the extensions $\Bbb Z_p\subset \Bbb F_q\subset F$. The extension $\Bbb Z_p\subset \Bbb F_q$ has degree $s$ (since $q=p^s$).
By assumption, the extension $\Bbb F_q\subset F$ has degree $n$. Thus, by Degree formula the extension $\Bbb Z_p\subset F$ has degree $sn$, so  $F$ is the field  $\Bbb F_{p^{sn}}$.
In the other words, $F=\Bbb F_{q^n}$, $F$ contains $q^n$ elements, and it is a splitting field of  the polynomial $x^{q^n} -x$ over the field $\Bbb Z_p$. Each element  $a\in \Bbb  F_q$ is a root of the polynomial $x^{q^n}-x=0$.
Indeed, for any element $a\in \Bbb F_q$ the identity $a^q=a$ holds. It implies that $(a^q)^q=a^q=a$. In the similar way, it implies  that $a^{q^n}= a^{q^{n-1}}=\dots =a^q=a$. The splitting field of the polynomial $x^{q^m}-x$ over the field $\Bbb F_q$ is equal to $\Bbb F_{q^m}$, since all elements of the field $\Bbb F_q$ are roots of that polynomial.
\end{proof}

As a corollary of Theorem \ref {findielex}, one can prove another proof of Lemma~\ref{subfield}.

\begin{proof}[Second proof  of Lemma \ref{subfield}]
The field $\Bbb F_{p^n}$ contains the subfield $\Bbb F_{p^k}$. Then the field $\Bbb F_{p^n}$ can be considered as an extension on the field $\Bbb F_{p^k}$. But all such extensions all described in Theorem \ref{findielex}. This theorem implies that $k$ is a divisor of $n$. In the opposite direction, if $k$ is a divisor of $n$, then the the field $\Bbb F_{p^n}$ obtains the field $\Bbb F_{p^k}$ follows from Theorem \ref{findielex} (or could be easily proven separately).
\end{proof}

As another  corollary of Theorem \ref {findielex}, one can
completely describe all extensions $\Bbb F_q\subset B$ which are contained in a field $F$, where $\Bbb F_q\subset F$ is an extension of degree $n$.

\begin{theorem}\label{cordeq} Let $\Bbb F_q\subset F$ be an extension of degree $n$ of the finite field $\Bbb F_q$, where $q=p^s$.
 Then for any divisor $k$ of the degree $n$, i.e., $n=km$ for some $m\in \Bbb z$,  there is a unique  intermediate field $B$
 such that the extension $\Bbb F_q\subset B$ has degree $k$.
 \end{theorem}

 \begin{proof} Theorem \ref{cordeq} in one direction follows from the Degree formula.  Indeed, let $B$ be an intermediate  field, i.e., the inclusions $\Bbb F_q\subset B \subset F$ holds. Assume that the degrees of extension $\Bbb F_q\subset B$ and $B\subset F$ are $k$ and $m$ respectively. By assumption  of the theorem, the degree of extension $\Bbb F\subset F$ is equal to $n$. The Degree formula implies that $n=km$.

 Theorem in the opposite direction follows from and Theorem \ref{findielex}. Indeed, by Theorem \ref{findielex},  there is a unique (up to a corresponding isomorphism) extension $B$ of degree $k$ of the field  $\Bbb F_q$. By the same  by Theorem \ref{findielex}, there is a unique (up to a corresponding isomorphism) extension $F_1$ of degree $m$ of the field $B$. The constructed extension $\Bbb F_q \subset F_1$, by theorem \ref{findielex}, is isomorphic to the extension $\Bbb F_q\subset F$. So this extension has an intermediate field with needed properties.
  \end{proof}

\section{Automorphisms of finite fields}\label{five}

In this section, we study automorphisms of finite fields. In particular, we show that any automorphism of a finite field $F$ of characteristic $p$ is a power of the Frobinius map $\Phi:F\to F$.

Let $F$ be a field of characteristic $p$. Recall that the Frobenius map  $\Phi:F\to F$ is the map which  sends an element $a\in F$ to the element $\Phi(a)=a^p$. The map $\Phi $ is a field homomorphism, i.e., it is a one-to-one map which respects arithmetic operations. In general, the map $\Phi$  is not an onto map.

\begin{theorem}\label{auto}  Let $F$ be the finite field $\Bbb F_{p^n}$.  Then:
\begin{enumerate}
\item the Frobenius map $\Phi:\Bbb F_{p^n}\to \Bbb F_{p^n}$ is an automorphism  of $\Bbb F_{p^n}$;
\item  an element $a\in  \Bbb F_{p^n}$ is fixed under   the map $\Phi^k$ if and only if $a$ is a root of the polynomial  $x^{p^k}-x$. In particular, the map $\Phi^k$ fixes at most $p^k$  elements of the field $\Bbb F_{p^n}$;

\item the map $\Phi^k$ fixes exactly  $p^k$ elements of the field $\Bbb F_{p^n}$ if and only if $k$ is a divisor of $n$. Moreover, if $k$ is a divisor of $n$, then the set of fixed elements  of the map $\Phi^k$ is the set of elements in the  subfield $\Bbb  F_{p^k}$.

\end{enumerate}
\end{theorem}

\begin{proof} 1) The map $\Phi:\Bbb F_{p^n}\to \Bbb F_{p^n}$ is an onto map, since it is one-to-one map and $\Bbb F_{p^n}$ is a finite set. Thus, $\Phi$ is an automorphism of the field $\Bbb F_{p^n}$.

2) Follows from an identity $\Phi^k(a)\equiv a^{p^k}$.

3)  A splitting field of the polynomial $x^{p^k}-x$ is contained in the field  $\Bbb F_{p^n}$ if and only if $k$ is a divisor of $n$. If $k$ is a divisor of $n$, then the set of all roots in $\Bbb F_{p^n}$ of the polynomial $x^{p^k}-x$ coincides with the set of elements of the subfield $\Bbb F_{p^k}$.
 \end{proof}

\begin{lemma}\label{bound} Any collection $G$ of automorphisms  of the finite field $\Bbb F_{p^n}$ contains at most $n$ elements.
\end{lemma}

\begin{proof} Let $a\in \Bbb F^*_{p^n}$  be a generator of the cyclic multiplicative group $\Bbb F^*_{p^n}$ of the field $\Bbb F_{p^n}$. Let $P_a\in \Bbb Z_p[x]$ be a minimal polynomial of $a$ over the field $\Bbb Z_p$. Since $a$ generates  the field  $\Bbb F_{p^n}$, the degree of the polynomial $P_a$ is equal to the degree of the field extension $\Bbb Z_p\subset \Bbb F_{p^n}$ which is equal to $n$. Any irreducible polynomial over the field  $\Bbb Z_p$ has simple roots only.  So the polynomial $P_a$ has at most $n$ different roots in the field $\Bbb F_{p^n}$. Any automorphism $g$ of the field $\Bbb F_{p^n}$ fixes pointwise all elements of the subfield $\Bbb Z_p$. Since all coefficients of the polynomial  $P_a$ belong to the field $\Bbb Z_p$, an automorphism $g$ maps the root $a$ of the polynomial $P_a$ to a root  of the same polynomial $P_a$. An automorphism $g$ is totally determined by the element $g(a)$,  since $a$ generates the multiplicative group  $\Bbb F^*_{p^n}$ of the field $\Bbb F_{p^n}$. Indeed,  if  $b\neq 0$ and $b=a^r$, then $g(b)=g^r(a)$, if $b=0$, then $g(b)=0$.
Since $P_a$ has at most $n$ different  roots in the field $\Bbb F_{p^n}$,  the collection of automorphisms  $G$  contains at most $n$ elements.
\end{proof}

\begin{theorem}\label{allauto} Any automorphism of the field $\Bbb F_{p^n}$  is a power $\Phi^k$ of the Frobenius map $\Phi$. All automorphisms of the field $\Bbb F_{p^n}$ form a cyclic group of order $n$ generated by $\Phi$.
\end{theorem}

\begin{proof} The set $A= \{I=\Phi^0, \Phi^1,\dots, \Phi^{n-1}\}$ contains  $n$ different automorphisms of the field $\Bbb F_{p^n}$. By Lemma \ref{bound}, all automorphisms of the field $\Bbb F_{p^n}$ belong to the set $A$. All automorphisms of  the field $\Bbb F_{p^n}$ form a cyclic group of order $n$ generated by $\Phi$, since $\Phi^n=I$.
\end{proof}

\section{Galois group of an extension of a finite field}\label{six}

In this section, we define the Galois group $G(F,K)$ of a finite extension $K\subset F$, where $K$ and $F$ are finite fields. We show that the group $G(F,K)$ is well defined, and  describe it.

\begin{definition}\label{normal} A finite  extension $K\subset F$ is {\it normal} if there is a finite group $G(F,K)$ of automorphisms  of the field $F$ such that an element $x\in F$ is fixed under the action of each element $g\in G(F,K)$, i.e., $g(a)=a$, if and only if $a\in K$. For a normal field extension $K\subset F$ the group $G(F,K)$ is called the {\it Galois group} of the extension.
\end{definition}

 One can show that the group $G(F,K)$ is well defined: a group of automorphisms of the field $F$ which fixes the subfield $K$ only, if it exists, is unique. Below, we show that any nested pair $K\subset F$  of finite fields is normal, and  its Galois  group is well defined. Moreover, we will completely describe this  group.

 First of all, Theorem \ref{allauto} implies the following corollary.

 \begin{corollary}\label{aalfirst} For any finite field $\Bbb F_{p^n}$ the field extension $\Bbb Z_p\subset \Bbb F_{^n}$ is normal, and its Galois group   is well defined. Moreover, the Galois group is the cyclic group of order $n$, generated by the Frobenius map $\Phi$.
 \end{corollary}

 Below, we generalize Corollary \ref{aalfirst} for any degree $n$ extension of any  finite  field.

 \begin{theorem}\label{galoiagroup} Let  $\Bbb F_q\subset \Bbb F_{q^n}$ be the degree $n$ extension of the field $\Bbb F_q$, where  $q=p^s$. Then this field extension  is normal, and its Galois group  is well defined. Moreover, the Galois group is the cyclic group of order $n$ generated by the $s$-th power $\Phi^s$ of the Frobenius map $\Phi$.
  \end{theorem}

 \begin{proof} Consider the field extension $\Bbb Z_p\subset \Bbb F_{q^n}$ of degree $sn$.  By Theorem \ref{aalfirst}, its Galois   group  is well defined. Moreover, the Galois group  is the cyclic group of order $sn$, generated by the Frobenius  map $\Phi $. Any group of automorphisms of the field $\Bbb F_{q^n}$ is a subgroup of the group $G(\Bbb F_{q^n},\Bbb Z_p$). In particular, any such group  is a cyclic group. Thus, the
 subgroup of all automorphisms of the field $\Bbb F_{q^n}$, which fixes the field $\Bbb F_q$, is a cyclic group. Let $\Phi^k$ be the smallest positive power of $\Phi$ which belongs to this subgroup. The power $k$ can not be smaller than $s$. Indeed,  the set of all fixed points of $\Phi^k$  contains  at most  $p^k$ elements, which is smaller than $q=p^s$. On the other hand, the subgroup  contains the map $\Phi^s$. Indeed, the set of fixed points of the map $\Phi^s$ is equal to $\Bbb F_{q}$. Thus, the Galois group $G(\Bbb F_{q^n},\Bbb F_q)$ is well defined. Moreover, it is the cyclic group on order $n$ generated by the map $\Phi^s$.
\end{proof}

\section{Galois correspondence for finite extension of finite field}\label{seven}
In this section, we discuss the Galois correspondence between the set of all subgroups  of the Galois Group $G(F,K)$ of an extension of degree $n$ of a finite field $K$ and  the set of all intermediate subfields  of the extension $K\subset F$.

Consider a degree $n$ extension $\Bbb F_q\subset \Bbb F_{q^n}$ of the finite field $\Bbb F_q$, where $q=p^s$. The intermediate fields for this extension and the subgroups of the  Galois  group of this extension are in one-to-one correspondence with the divisors $m$ of the number $n$.  Indeed, the following lemma holds.

\begin{lemma}  With any divisor $m$ of $n$ such that $n=mk$ there exists:
\begin{enumerate}                                              \item the unique intermediate  field $B$ such that the degree
  of the extension $B\subset \Bbb F_{q^n}$ is equal to
  $m$; for such field $B$ the degree of the extension $\Bbb F_q\subset B$ is equal to $k$;

\item the unique subgroup $G$ in the
Galois  group   $G(\Bbb F_q,\Bbb F_{q^n})\sim \Bbb Z_n$
which is a cyclic group of order $m$.  For such a group $G$
the factor group  $G(\Bbb F_q,\Bbb F_{q^n})/G$  is
the  cyclic group of order $k$.
\end{enumerate}
\end{lemma}

\begin{proof}  By theorem \ref{cordeq}, for any divisor $k$ of $n$ such that $n=k m$, there is a unique intermediate field $B$. $\Bbb F_q\subset B\subset \Bbb F_{q^n}$ such that the degree of the extension $\Bbb F_q\subset B$ is equal to $k$. For such extension $B$ the degree of the extension $B\subset \Bbb F_{q^n}$ is equal to $m$, since the degree of the extension $\Bbb F_q\subset \Bbb F_{q^n}$ is equal to $n=km$.

By Theorem \ref{galoiagroup},  the Galois group of the extension $\Bbb F_q \subset \Bbb F_{q^n}$ is the cyclic group $\Bbb Z_n$  of order $n$. The cyclic group of order $n$ contains a unique subgroup $G$ of order $m$ (and the order of any subgroup divides the order of group). The factor group of $\Bbb Z_n$ by a cyclic group of order $m$ is the cyclic group of order $k$.
  \end{proof}

The Galois correspondence provides a natural one-to-one
correspondence   between the set of intermediate fields in a normal field extension and the set of all subgroups of the Galois group of
 the extension.
\begin{theorem}[Galois correspondence for extensions of finite fields] \label{Galcor} Let $\Bbb F_q\subset \Bbb F_{q^n}$ be the degree $n$ extension of the  field $\Bbb F_q$. Then the Galois correspondence provides a bijection between the set of all subgroups of the Galois group $G(\Bbb F_{q^n},  \Bbb F_q)$ and the set of all intermediate subfields of the extension.

Moreover, if an intermediate field $B$  corresponding to a subgroup $G$, then the factor group $G(\Bbb F_{q^n},\Bbb F_q)/G$ acts naturally on the field $B$, and the field of all fixed  elements  of this action on $B$ is the field $\Bbb F_q$.
\end{theorem}

\begin{proof}Galois group of the extension $\Bbb F_q \subset \Bbb F_{q^n}$ is the cyclic group of order $n$ generated by the map $\Phi^s$. The intermediate fields $B$, such that the degree of the extension $\Bbb F_q\subset B$ is equal to $k=n:m$, is the field  $B=\Bbb  F_{q^k}$.  The Galois group of the extension $\Bbb  F_{q^k}\subset \Bbb F_{q^n}$ is the cyclic group $G$ of order $m=n:k$, generated by the map $\Phi^{sk}$. Since $G$ is the Galois group of the extension $B\subset \Bbb F_{q^n}$, it fixes pointwise all elements of the field $B$, and it fixes elements of $B$ only.

One can describe the factor group $(\Bbb F_{q^n},\Bbb F_q)/G$  as a factor group of a cyclic group of order $n$  generate by the map $\Phi^s$ by its cyclic subgroup generated by the map $\Phi ^{sk}$. Such factor group is a cyclic group of order $m=n:k$
which is the Galois group of the extension $\Bbb F_{q^{sk}}\subset \Bbb F_{q^n}$, i.e., it is the Galois group
$G(B,\Bbb F_{p^n})$.
\end{proof}

\section{Galois theory}\label{seceight}
In this section, we present, without proof, basic theorems of Galois theory which widely generalize the  results on extensions of finite fields which we discuss above. Proofs of these theorems will be presented in the course later.

\begin{definition} An element $a\in E$ is called a {\it primitive element} for a finite field extension $K\subset E$   if $E$ can be obtained by adjoining $a$ to the field $K$.
\end{definition}

\begin{theorem}[Primitive Element Theorem]\label{primel} For any separable finite field extension $K\subset E$ there exists  a primitive element $a\in E$.
\end{theorem}

Theorem \ref{primel}  widely  generalized
Corollary \ref{finprimel} which plays a similar role for extensions of finite fields.

Above we defined normal field  extensions and their  Galois group (see  Definition \ref{normal}).

\begin{theorem} A finite field extension $K\subset F$ is normal, if and only if it is separable and the field $F$ is a splitting field of some polynomial over $K$. The Galois group $G(K,F)$ of the normal extension is well defined, i.e., there   is a unique group of automorphisms $G(K,F)$ of $F$ with fixes an  element $a\in F$ if and only if  $a\in K$.
\end{theorem}

Let $K \subset F$ be a field extension of degree $n$. For an intermediate field $B$ we denote  by $k$ and $m$ degrees of the extensions $K \subset B$ and $B \subset F$, respectively. Then $k$ and $m$  are divisors of $n$, and  $n = km$.
However, in general, it is not true that for every divisor $k$ of $n$ there exists a unique intermediate field $B$ such that the degree of the extension $K\subset B$ is equal to $k$. There may be several such subfields, or none at all.

Similarly, the Galois group of a normal extension of degree $n$ has order $n$, and, by Lagrange's theorem,  orders of its  subgroups divide $n$. Yet, in general, for a given divisor $k$ of $n$, there need not be a unique subgroup of order $k$; there may be several such subgroups, or none.

Nevertheless, the following Galois Correspondence Theorem holds for every normal field extension.

\begin{theorem}[Galois correspondence] \label{genGalcor} Let $K\subset F$ be a normal extension of degree $n$. Then the Galois correspondence provides a bijection between the set of all subgroups of the Galois group $G(F,K)$ and the set of all intermediate subfields of the extension.

Moreover, for an intermediate field $B$ the extension $K\subset B$ is normal if and only if $B$ corresponds to a normal divisor $G$ of the Galois group $G(K,F)$. In that case, there is a natural action  of  the factor group $G(F,K)/G$
on the intermediate field $B$ which  fixes all elements of the field $K$. It provides an isomorphism between the factor group $G(F,K)/G$  and the
Galois group $G(K,B)$ of the extension $K\subset B$.
 \end{theorem}

\end{document}